\documentclass{amsart}

\usepackage{color,graphicx,amssymb,latexsym,amsfonts,txfonts,amsmath,amsthm}
\usepackage{pdfsync}
\usepackage{amsmath,amscd}
\usepackage[all,cmtip]{xy}

\usepackage{hyperref}
\hypersetup{
    colorlinks=true,       % false: boxed links; true: colored links
    linkcolor=blue,          % color of internal links
    citecolor=blue,        % color of links to bibliography
    filecolor=blue,      % color of file links
    urlcolor=blue           % color of external links
}

 \newtheorem{theo}{Theorem}[section]
 
 \newtheorem{lemm}[theo]{Lemma}
 \newtheorem{prop}[theo]{Proposition}
 \newtheorem{conj}[theo]{Question}
 \theoremstyle{definition}
 
 \theoremstyle{remark}
 \newtheorem{rem}[theo]{\bf Remark}
 \newtheorem{ex}{\bf Example}
 \numberwithin{equation}{section}

\begin{document}

\title{Computing the Field of moduli of some non-hyperelliptic pseudo-real curves}

\author{Rub\'en A. Hidalgo}
\address{Departamento de Matem\'atica y Estad\'{\i}stica, Universidad de La Frontera,  Temuco, Chile}
\email{ruben.hidalgo@ufrontera.cl}

\begin{abstract}
The explicit computation of the field of moduli of a closed Riemann surface is, in general, a difficult task. In this paper, for each even integer $k \geq 2$, we consider a 
suitable $2$-real parameter family of non-hyperelliptic pseudo-real Riemann surfaces of genus $g=1+(2k-3)k^{4}$. For each of them, we compute its field of moduli and also a minimal field of definition.
\end{abstract}

\thanks{Partially supported by project Fondecyt 1230001}

\keywords{Riemann Surfaces, Algebraic curves, Automorphisms, Fields of moduli} 
\subjclass[2010]{Primary 14H37, 14H45; Secondary 30F10}

\maketitle

%%%%%%%%%%%%%%%%%%%
%%%%%%%%%%%%%%%%%%%
\section{Introduction}
Fields of moduli were first introduced by Matsusaka \cite{Matsusaka} and later by Shimura \cite{Shimura} (for the case of polarized abelian varieties). In \cite{Koizumi}, Koizumi gave a more general definition for general algebraic varieties (even with extra structures). In general, the explicit computation of the field of moduli is a difficult task.
In this paper, we are concerned with explicit computations of the field of moduli in the case of algebraic curves.

Let $S$ be a closed Riemann surface of genus $g \geq 2$. Due to the Riemann-Roch theorem, $S$ can be realized by smooth irreducible complex projective algebraic curves, so we may consider its field of moduli ${\mathcal M}(S)$ (see Section \ref{Sec:FOM}).
A subfield of ${\mathbb C}$ is called a field of definition of $S$ if it is possible to find one of such curves defined over it. By \cite{Koizumi},  
it is known that $S$ has a field of definition which is a finite extension of ${\mathcal M}(S)$, and that ${\mathcal M}(S)$ coincides with the intersection of all those fields of definitions of $S$.
Weil's Galois descent theorem \cite{Weil} provides a sufficient condition for a subfield of ${\mathbb C}$ to be a field of definition of $S$. Such a result, in particular, asserts that if $S$ has a trivial group of automorphisms, then $S$ is definable over ${\mathcal M}(S)$.
In  \cite{Wo1,Wo2}, Wolfart noted that if $S/{\rm Aut}(S)$ has genus zero and exactly three cone points (i.e., $S$ is quasiplatonic), then again $S$ is definable over ${\mathcal M}(S)$. In \cite{AQ}, Artebani and Quispe proved that if $S/{\rm Aut}(S)$ has genus zero and has an odd number of cone points with the same cone order, then $S$ is also definable over ${\mathcal M}(S)$. 
If $S$ has a non-trivial group of automorphisms such that $S/{\rm Aut}(S)$ has genus zero, then $S$ admits a field of definition that is a quadratic extension of ${\mathcal M}(S)$ \cite{Hidalgo:FOD/FOM}. 

Explicit examples of hyperelliptic Riemann surfaces (Riemann surfaces admitting a two-fold-branched cover over the Riemann sphere) that cannot be defined over their field of moduli were provided by Earle in \cite{Earle1,Earle2} and by Shimura in \cite{Shimura} and later by Hugging in \cite{Huggins2}. Examples of non-hyperelli[tic Riemann surfaces with the same property were constructed in \cite{Hidalgo,Hidalgo2} and by Kontogeorgis in \cite{Kontogeorgis}. Later, other examples were obtained in \cite{AQ, Achq}.
The main idea, in all these examples, was to prove that the field of moduli is a subfield of ${\mathbb R}$ (which is equivalent for the Riemann surface to have an anticonformal automorphism) but that they cannot be defined over ${\mathbb R}$ (which is equivalent for the surface to have no anticonformal involutions \cite{Silhol}). These kinds of Riemann surfaces are called pseudo-reals. In those examples, it was not necessary to compute explicitly the field of moduli (it was only needed to construct anticonformal automorphisms and to prove that none of them can be chosen to be an involution). In \cite{BHQ}, it was constructed a family of (both, hyperelliptic and non-hyperelliptic) Riemann surfaces which are definable of ${\mathbb R}$ but which cannot be defined over 
the field of moduli.

In \cite{KFT}, it was computed the field of moduli of those non-hyperelliptic Riemann surfaces $S$ of genus three admitting a group $G$, isomorphic to the symmetric group $S_{4}$, as a group of automorphisms (in this case, the quotient $S/G$ has genus zero and four cone points of orders $2$, $2$, $2$ and $3$, so they are defined over their field of moduli). In \cite{HJ}, it was computed the field of moduli of those non-hyperelliptic Riemann surfaces $S$ admitting a group of automorphisms $H \cong {\mathbb Z}_{k}^{3}$, where $k \geq 3$ is an integer, such that $S/H$ is of genus zero and exactly four cone points, each one of order $k$ (these are non-hyperelliptic and definable over their fields of moduli).

In this paper, for each even integer $k \geq 2$, we construct a two-real dimensional family of non-hyperelliptic pseudo-real curves $C^{(k)}_{r,\theta}$ of genus $g=1+(2k-3)k^{4}$  (generalizing those constructed in \cite{Hidalgo,Hidalgo2}; corresponding to $k=2$). The curve $C^{(k)}_{r,\theta}$ admits a group $H \cong {\mathbb Z}_{k}^{5}$ of conformal automorphisms such that $C^{(k)}_{r,\theta}/H$ has genus zero and exactly six cone points, each one of order $k$. For each of them (see Theorem \ref{teo1}), we compute explicitly their field of moduli and also a minimal field of definition (which is necessarily a quadratic extension of the corresponding field of moduli).

\smallskip

This paper is structured as follows. In Section \ref{sec:prelim}, we recall the definitions of field of moduli, field of definition, and minimal field of definition, for the case of curves. In Section \ref{Sec:GFC}, we introduce the $2$-real dimensional family of pseudo-real curves $C^{(k)}_{r,\theta}$. In Section \ref{Sec:computation}, for each of these curves $C^{(k)}_{r,\theta}$, we describe those field automorphisms of ${\mathbb C}$ which determine their field of moduli (Proposition \ref{teo0}), and compute the field of moduli and a minimal field of definition of some of them (Theorem \ref{teo1}). In Section \ref{sec:examples}, we provide two explicit examples. Section \ref{Sec:pruebapropo} contains the proof of Proposition \ref{teo0}.

%%%%%%%%%%%%%%%%%%%
\section{Preliminaries}\label{sec:prelim}

\subsection{Field of moduli and fields of definition of curves}
Let $C$ be a smooth and irreducible complex projective algebraic curve, say 
defined as the set of common zeros of the homogeneous polynomials $P_{1},\ldots,P_{r} \in {\mathbb C}[x_{1},\ldots,x_{m}]$. In particular, $C$ defines a closed Riemann surface of some genus $g$. 

%%%%%%%%%%%%%
\subsubsection{The field of moduli of $C$}\label{Sec:FOM}
If $\sigma \in {\rm Aut}({\mathbb C}/{\mathbb Q})$ (the group of field automorphisms of ${\mathbb C}$), then we may consider the new polynomials $P_{1}^{\sigma}$,..., $P_{r}^{\sigma}$, where the coefficients of $P_{j}^{\sigma}$ are the corresponding images under $\sigma$ of the coefficients of the original polynomial $P_{j}$. The algebraic curve $C^{\sigma}$, defined by these new polynomials, is still a smooth and irreducible complex projective algebraic curve with the same algebraic invariants as $C$. In particular,  $C^{\sigma}$ is again a closed Riemann surface of genus $g$, but it might be that $C$ and $C^{\sigma}$ define non-conformally equivalent Riemann surfaces. 
Let $G_{C}<{\rm Aut}({\mathbb C}/{\mathbb Q})$ be the group defined by those $\sigma \in {\rm Aut}({\mathbb C}/{\mathbb Q})$ with the property that $C$ and $C^{\sigma}$ define conformally equivalent Riemann surfaces. The fixed field of $G_{C}$, denoted as ${\mathcal M}(C)$, is called the {\it field of moduli} of $C$.
In general, the computation of this field is not an easy task.

\begin{rem}
Note that if $C_{1}$ and $C_{2}$ are smooth and irreducible complex projective algebraic curves, defining conformally equivalent  Riemann surfaces, then for each $\sigma \in {\rm Aut}({\mathbb C}/{\mathbb Q})$ it holds that $C_{1}^{\sigma}$ and $C_{2}^{\sigma}$ are still conformally equivalent Riemann surfaces. In particular, the above provides an action of the group ${\rm Aut}({\mathbb C}/{\mathbb Q})$ over the moduli space ${\mathcal M}_{g}$  of genus $g$. If we denote by  $[C]$ the corresponding class of $C$ in ${\mathcal M}_{g}$, then $G_{C}$ is the stabilizer of $[C]$ with respect to the action of ${\rm Aut}({\mathbb C}/{\mathbb Q})$ on ${\mathcal M}_{g}$. 
\end{rem}

%%%%%%%%%%%%%%%%
\subsubsection{Fields of definitions of $C$}
A subfield $K<{\mathbb C}$ is called a {\it field of definition} of $C$ if there is a smooth and irreducible complex projective algebraic curve $D$, defining a Riemann surface conformally equivalent to that of $C$, so that $D$ is defined by polynomials with coefficients in $K$. Notice from the definition that in this case ${\mathcal M}(C)={\mathcal M}(D)$. It is known that ${\mathcal M}(C)$ is a subfield of any field of definition of $C$ and that it is the intersection of all the fields of definitions of $C$ \cite{Koizumi}. Moreover, there is always a field of definition of $C$ being a finite extension of ${\mathcal M}(C)$ \cite{HH}. This fact permits to define a {\it minimal field of definition} for the curve $C$ as a field of definition so that the degree of the extension over ${\mathcal M}(C)$ is minimal along all fields of definitions of $C$  (minimal fields of definitions are not unique, but if ${\mathcal M}(C)$ is a field of definition of $C$, then it is the unique minimal field of definition of $C$). 

%%%%%%%%%%%%%%%%%
\subsection{Weil's Galois descent theorem}
\begin{theo}[Weil's Galois descent theorem \cite{Weil}]\label{Prop:Weil}
Let $C$ be a  smooth and irreducible complex projective algebraic curve defined over a finite Galois extension ${\mathbb L}$ of a field ${\mathcal M}$. 
If for every $\sigma \in {\rm Aut}({\mathbb L}/{\mathcal M})$ there is an isomorphism $f_{\sigma}:C \to C^{\sigma}$, defined over ${\mathbb L}$, such that for all $\sigma, \tau \in {\rm Aut}({\mathbb L}/{\mathcal M})$ the compatibility condition $f_{\tau\sigma}=f^{\tau}_{\sigma} \circ f_{\tau}$
holds (called a Weil datum), then there exists a smooth projective algebraic curve $E$ defined over ${\mathcal M}$ and there exists a biholomorphism $Q:C \to E$, defined over ${\mathbb L}$, such that $Q^{\sigma} \circ f_{\sigma}=Q$.
\end{theo}

In \cite{DE,HH}, it was proved that a smooth projective algebraic curve can be defined over a  finite Galois extension of its field of moduli (in most cases, such an extension can be assumed quadratic \cite{Hidalgo:FOD/FOM}).
This observation, together with Weil's Galois descent theorem, provides sufficient conditions for the field of moduli to be a field of definition, i.e., 
for complex curves of genus at least two, we may work with ${\rm Aut}({\mathbb C}/{\mathbb Q})$ and the extension ${\mathbb Q}<{\mathbb C}$ instead of restricting to finite Galois extensions.

\begin{theo}
Let $C$ be a  smooth and irreducible complex projective algebraic curve of genus at least two. Let 
${\mathcal M}$ be a subfield of ${\mathbb C}$ contained in the field of definition of $C$. Let us assume that 
for every $\sigma \in {\rm Aut}({\mathbb C}/{\mathcal M})$ there is a biholomorphism $f_{\sigma}:C \to C^{\sigma}$ such that for all $\sigma, \tau \in {\rm Aut}({\mathbb C}/{\mathcal M})$ the compatibility condition $f_{\tau\sigma}=f^{\tau}_{\sigma} \circ f_{\tau}$
holds. Then ${\mathcal M}$ is a field of definition of $C$.
\end{theo}
\begin{proof}
 As $C$ has a finite number of automorphisms, it can be checked that each $f_{\sigma}$ is defined over a finite extension of ${\mathcal M}$. In this way, we may find a finite Galois extension ${\mathbb L}$ of ${\mathcal M}$ over which $C$ and all the isomorphisms $f_{\sigma}$ are defined. Then we may use Weil's Galois descent theorem to ensure that there exists a smooth projective algebraic curve $E$ defined over ${\mathcal M}$ and there exists a biholomorphism $Q:C \to E$, defined over ${\mathbb L}$, such that $Q^{\sigma} \circ f_{\sigma}=Q$. 
\end{proof}

%%%%%%%%%%%%%%%%%%%%
\subsection{Generalized circles and cross-ratio}\label{cross-ratio}
A generalized circle in the Riemann sphere $\widehat{\mathbb C}$ is either a Euclidian circle in ${\mathbb C}$ or the union of $\infty$ with an Euclidian line in ${\mathbb C}$.
Given four different points $a,b,c,d \in \widehat{\mathbb C}$, it is defined the cross-ratio $[a,b,c,d]=T(d)$, where $T$ is a the unique M\"obius transformation satisfying that $T(a)=\infty$, $T(b)=0$ and $T(c)=1$. By the definition, $[a,b,c,d] \in {\mathbb C}-\{0,1\}$.
If $M$ is any M\"obius transformation, then $[M(a),M(b),M(c),M(d)]=[a,b,c,d]$. 
The points $a,b,c,d$ belong to a common generalized circle if and only if $[a,b,c,d] \in {\mathbb R}$. In particular, M\"obius transformations send generalized circles into generalized circles.
Let $${\mathbb G}=\langle A(z)=1/z, B(z)=z/(z-1)\rangle \cong {\mathfrak S}_{3}\; \mbox{\rm (the symmetric group in three letters)}.$$

Any permutation of the four points changes the value of $[a,b,c,d]$ to a value $R([a,b,c,d])$, where $R \in {\mathbb G}$. In particular, if $[a,b,c,d] \in \{-1,1/2,2\}$ then the cross-ratio of any permutation of these four points is still in the same set.
If $a \neq 0, \infty$, then $[\infty,0,a,-a]=-1$. The only cross-ratios, obtained by permutation of $\infty$, $0$, $a$ and $-a$, producing the same value of $-1$ are given by $[\infty,0,a,-a]$, $[\infty,0,-a,a]$, $[0,\infty,a,-a]$, $[0,\infty,-a,a]$, $[a,-a,\infty,0]$, $[-a,a,\infty,0]$, $[a,-a,0,\infty]$ and $[-a,a,0,\infty]$.

%%%%%%%%%%%%%%%%%%
\subsection{A parameter space}
Let us denote by $\Omega \subset {\mathbb C}^{3}$ the region consisting of those triples $(\lambda_{1},\lambda_{2},\lambda_{3})$ such that (i) $\lambda_{j} \in {\mathbb C}\setminus \{0,1\}$, and (ii) $\lambda_{i} \neq \lambda_{j}$ if $i \neq j$. Some holomorphic automorphisms of $\Omega$ are given by the following ones  
$$A(\lambda_{1},\lambda_{2},\lambda_{3})=\left(\frac{\lambda_{3}}{\lambda_{3}-1},\frac{\lambda_{3}}{\lambda_{3}-\lambda_{1}},\frac{\lambda_{3}}{\lambda_{3}-\lambda_{2}}\right), \quad B(\lambda_{1},\lambda_{2},\lambda_{3})=(\lambda_{1}^{-1},\lambda_{2}^{-1},\lambda_{3}^{-1}).$$
It can be checked that  ${\mathbb U}=\langle A,B \rangle \cong {\mathfrak S}_{6}$, the symmetric group in six letters.

%%%%%%%%%%%%%%%%%%%%%%%%
%%%%%%%%%%%%%%%%%%%%%%%%
\section{Generalized Fermat curves of type $(k,5)$}\label{Sec:GFC}
A closed Riemann surface $S$ is called a {\it generalized Fermat curve of type $(k,5)$}, where $k \geq 2$ is an integer, if it admits a group $H \cong {\mathbb Z}_{4}^{5}$ of conformal automorphisms such that the quotient orbifold $S/H$ has signature $(0;k,k,k,k,k,k)$, that is, of genus zero and with exactly six cone points, each one of order $k$. We say that $H$ is a {\it generalized Fermat group of type $(k,5)$} and that $(S,H)$ is a {\it generalized Fermat pair of type $(k,5)$}. In \cite{GHL}, it was observed that $S$ is non-hyperelliptic. By the Riemann-Hurwitz formula, $S$ has genus $g=1+(2k-3)k^{4}$.

Let $(S,H)$ be a generalized Fermat pair of type $(k,5)$ and let $\pi:S \to \widehat{\mathbb C}$ be a holomorphic branched Galois covering with deck group $H$.
Up to a M\"obius transformation, we may assume these six branched values of $\pi$ to be $\infty$, $0$, $1$, 
$\lambda_{1},\lambda_{2},\lambda_{3}$, where $(\lambda_{1},\lambda_{2},\lambda_{3}) \in \Omega$. It is known (\cite{CGHR,GHL,Hidalgo:Homology}) that (i) $S$ can be algebraically described by the  smooth irreducible projective algebraic curve 
$$C^{(k)}(\lambda_{1},\lambda_{2},\lambda_{3})=\left\{ \begin{array}{ccc}
x_{1}^{k} + x_{2}^{k} + x_{3}^{k}&=&0\\
\lambda_{1} x_{1}^{k} + x_{2}^{k} + x_{4}^{k}&=&0\\
\lambda_{2} x_{1}^{k} + x_{2}^{k} + x_{5}^{k}&=&0\\
\lambda_{3} x_{1}^{k} + x_{2}^{k} + x_{6}^{k}&=&0
\end{array}
\right\}
\subset {\mathbb P}^5
$$
(ii) $H=\langle a_{1},a_{2},a_{3},a_{4},a_{5}\rangle$, where 
$a_{j}$ is the linear transformation given by multiplication of the $x_{j}$-coordinate by $e^{2 \pi i/k}$, and (iii) $\pi$, in the above model, is defined by 
$$\pi:C_{r,\theta} \to \widehat{\mathbb C}: [x_{1}:\cdots:x_{6}] \mapsto -\left( \frac{x_{2}}{x_{1}} \right)^{k}.$$

\begin{rem}
In \cite{Hidalgo:JNT}, it was noted that there exist six holomorphic differential forms, $\Theta_{1},\ldots,\Theta_{6} \in {\rm H}^{1,0}(S)$, such that the map
$F:S \to {\mathbb P}^{5}$, given by $F(x)=[\Theta_{1}(x):\dots:\Theta_{6}(x)]$, is a holomorphic embedding whose image is a curve as above.
\end{rem}

In \cite{HKLP}, it was noted that every generalized Fermat curve of type $(k,5)$ admits a unique generalized Fermat group of that type. As a consequence of this uniqueness, we may obtain the following characterization for those $\sigma \in G_{C^{(k)}(\lambda_{1},\lambda_{2},\lambda_{3})}$.

\begin{lemm}\label{lemita1}
The field of moduli of $C^{(k)}(\lambda_{1},\lambda_{2},\lambda_{3})$ is the fixed field of the subgroup of ${\rm Aut}({\mathbb C}/{\mathbb Q})$ consisting of those $\sigma$ such that there is a M\"obius transformation $T_{\sigma}$ for which $\{T_{\sigma}(\infty),T_{\sigma}(0),T_{\sigma}(1),T_{\sigma}(\lambda_{1}),T_{\sigma}(\lambda_{2}),T_{\sigma}(\lambda_{3})\}=\{\infty,0,1,\sigma(\lambda_{1}),\sigma(\lambda_{2}),\sigma(\lambda_{3})\}$.
\end{lemm}
\begin{proof}
The uniqueness property of $H$ permit us to obeserve (see also \cite{GHL}) that $C^{(k)}(\lambda_{1},\lambda_{2},\lambda_{3})$ and $C^{(k)}(\hat{\lambda}_{1},\hat{\lambda}_{2},\hat{\lambda}_{3})$ are conformally equivalent if and only if there exists some $U \in {\mathbb U}$ such that $U(\lambda_{1},\lambda_{2},\lambda_{3})=(\hat{\lambda}_{1},\hat{\lambda}_{2},\hat{\lambda}_{3})$.
\end{proof}

\begin{rem}
The above result permits us to observe that ${\mathcal M}(C^{(k)}(\lambda_{1},\lambda_{2},\lambda_{3}))$ is also the field of moduli of the genus two Riemann surface $y^{2}=x(x-1)(x-\lambda_{1})(x-\lambda_{2})(x-\lambda_{3})$.

\end{rem}

%%%%%%%%%%%%%%%%%%%%%%%%%%
\subsection{The two-real-dimensional family $C^{(k)}_{r,\theta}$}
As, for the generic tuple $(\lambda_{1},\lambda_{2},\lambda_{3}) \in \Omega$, the group of M\"obius transformations keeping invariant the set $\{\infty,0,1,\lambda_{1},\lambda_{2},\lambda_{3}\}$ is trivial, the uniqueness of $H$ also permits to observe that, for such a generic tuple, one has ${\rm Aut}(C^{(k)}(\lambda_{1},\lambda_{2},\lambda_{3}))=H$. 
Now, assume ${\rm Aut}(C^{(k)}(\lambda_{1},\lambda_{2},\lambda_{3}))=H$. If $Q$ is an anticonformal automorphism of $C^{(k)}(\lambda_{1},\lambda_{2},\lambda_{3})$, then $Q^{2} \in H$ and every anticonformal automorphism of such a curve is of the from $h \circ Q$, where $h \in H$.
Below, we proceed to describe a two-real-dimensional family of those curves satisfying the above properties.

%%%%%%%%%%%%%%%
\subsubsection{}
If $\theta, r$ are real numbers, $r>1$, then we set $$C^{(k)}_{r,\theta}=C^{(k)}(-r^{2},re^{i\theta},-re^{i\theta}).$$

\begin{rem}
(1) The curves $C^{(k)}_{r,\theta}$ and $C^{(k)}_{r,\theta+l\pi}$, where $l \in {\mathbb Z}$, are conformally equivalent Riemann surfaces.
As $C_{r,\theta}=C_{r,\theta+2l\pi}$, we only need to consider the case when $l$ is odd. In this case, 
the transformation 
$T_{1}([x_{1}:\cdots:x_{6}])=[x_{1}:x_{2}:x_{3}:x_{4}:x_{6}:x_{5}]$ defines a conformal equivalence between $C_{r,\theta}$ and $C_{r,\theta+l\pi}$.
(2) The curves $C^{(k)}_{r,\theta}$ and $C^{(k)}_{r,-\theta+l\pi}$, where $l \in {\mathbb Z}$, are anticonformally equivalent Riemann surfaces.
The transformation 
$T_{2}([x_{1}:\cdots:x_{6}])=[\overline{x_{1}}:\cdots: \overline{x_{6}}]$ defines an anticonformal equivalence between $C^{(k)}_{r,\theta}$ and $C^{(k)}_{r,-\theta+l\pi}$.
\end{rem}

%%%%%%%%%%%%%%%
\subsubsection{}
Let us define $r_{1}(\theta)=\sqrt{1+\cos(\theta)^{2}}-\cos(\theta)$ and $r_{2}(\theta)=\sqrt{1+\cos(\theta)^{2}}+\cos(\theta)$ (so 
$r_{1}(\theta)r_{2}(\theta)=1$). If we assume that $e^{i\theta} \notin \{\pm 1, \pm i\}$, then $r_{j}(\theta) \neq 1$ and 
$$P_{\theta}=\{-(\sqrt{1+\cos(\theta)^{2}}-\cos(\theta))^{2}, -(\sqrt{1+\cos(\theta)^{2}}+\cos(\theta))^{2}\} \cap (-\infty,-1)$$
has cardinality $1$; say $P_{\theta}=\{r_{\theta}\}$. For instance, $P_{\pi/3}=\{-(3+\sqrt{5})/2\}$.

\begin{lemm}\label{trivialgroup}
If $\theta$ is so that $e^{i\theta} \notin \{\pm 1, \pm i\}$, and $r>1$ is so that $r \neq r_{\theta}$, then 
${\rm Aut}(C^{(k)}_{r,\theta})=H$.
\end{lemm}
\begin{proof}
By direct inspection at the cross-ratios, with the extra restriction that $e^{i\theta} \notin \{\pm 1, \pm i\}$, we may notice that the only subsets of cardinality $4$ of the set 
$$\{\infty, 0, 1, -r^{2}, re^{i\theta}, -re^{i\theta}\}$$
contained in a generalized circle are given by
$$\{\infty,0,1,-r^{2}\}, \quad \{\infty,0,re^{i\theta}, -re^{i\theta}\}, \quad \{1,-r^{2}, re^{i\theta}, -re^{i\theta} \}.$$

Their corresponding cross-ratios  are given by
$$[\infty,0,1,-r^{2}]=-r^{2} \notin \{-1,1/2,2\}$$
$$[\infty,0,re^{i\theta}, -re^{i\theta}]=-1$$
$$[1,-r^{2},re^{i\theta}, -re^{i\theta}]=C_{\theta}(-r^2)=-\frac{r^{4}+2(2\sin(\theta)^{2}-1)r^2+1)}{(r^{2}+2\cos(\theta) r+1)^{2}} \notin \{-1,1/2,2\}$$

Now,  for $\theta$ fixed and each $R \in {\mathbb G}$, the equation $C_{\theta}(-r^{2})=R(-r^{2})$ determines a non-constant real polynomial, of degree either $2$ or $4$, for which $r$ is a zero. These polynomials are the following ones:
$$r^{2}+2\cos(\theta)r-1; \; r^{2}-2\cos(\theta)r-1; \; 2r^{4}+3r^{2}-2\cos(\theta)r+1;$$
$$2r^{4}+3r^{2}+2\cos(\theta)r+1; \; r^{4}+2\cos(\theta)r^{3}+3r^{2}+2; \; r^{4}-2\cos(\theta)r^{3}+3r^{2}+2.$$

Each one of the degree four polynomials has no real zeroes greater than $1$. Both degree two polynomials have zeroes, respectively, at the points $r_{1}(\theta)$ and $r_{2}(\theta)$.  
The condition that $r \neq r_{\theta}$ then asserts that the cross-ratios of the three collections of $4$-points as above are non-equivalent under the action of ${\mathbb G}$.  In particular,  if $T$ is a M\"obius transformation so that 
$$\{\infty,0,1,-r^{2},re^{i\theta}, -re^{i\theta}\} \stackrel{T}{\mapsto}\{\infty,0,1,-r^{2},re^{i\theta}, -re^{i\theta}\},$$
then necessarily 
$$\{\infty,0,1,-r^{2}\} \stackrel{T}{\mapsto} \{\infty,0,1,-r^{2}\}, \quad \{\infty,0,re^{i\theta}, -re^{i\theta}\} \stackrel{T}{\mapsto} \{\infty,0,re^{i\theta}, -re^{i\theta}\}.$$

We now have only two possibilities.
\smallskip
\begin{enumerate}
\item $T(\infty)=\infty$, $T(0)=0$ and either
\begin{enumerate}
\item $T(1)=1$, in which case $T(z)=z$; or 
\item $T(1)=-r^{2}$, in which case $T(z)=-r^{2} z$. But in this case, $T(re^{i\theta})=-r^{3}e^{i\theta}$ should be equal to $\pm re^{i\theta}$, that is, that $r^{2}=\pm 1$, a contradiction as $r>1$.
\end{enumerate}
\smallskip
\item $T(\infty)=0$, $T(0)=\infty$ and either
\begin{enumerate}
\item $T(1)=1$, in which case $T(z)=1/z$. But in this situation, we must have  that  $T(re^{i\theta})=r^{-1}e^{-i\theta}$ is equal to $\pm re^{i\theta}$, that is, $r^{2}e^{2i\theta}=\pm 1$, again a  contradiction as $r >1$.
\item $T(1)=-r^{2}$, in which case $T(z)=-r^{2}/z$. In this situation, we must have  that  $T(re^{i\theta})=-re^{-i\theta}$ is equal to $\pm re^{i\theta}$, that is, $e^{2i\theta}=\pm 1$, a contradiction as $e^{i\theta} \notin {\mathbb R} \cup\{ \pm i\}$.

\end{enumerate}
\end{enumerate}

If $A \in {\rm Aut}(C^{(k)}_{r,\theta}) \setminus H$, then (by the uniqueness of $H$), it will induce a M\"obius transformation $T$ as above, a contradiction. So, ${\rm Aut}(C^{(k)}_{r,\theta})=H$.
\end{proof}

\begin{lemm}\label{anticonformal}
If $\theta$ is so that $e^{i\theta} \notin \{\pm 1, \pm i\}$, and $r>1$ is so that $r \neq r_{\theta}$, then $C^{(k)}_{r,\theta}$ admits anticonformal automorphisms (so, ${\mathcal M}(C^{(k)}_{r,\theta})<{\mathbb R}$) and all of them are of the form
$$
Q([x_1:x_2:x_3:x_4:x_5:x_6])=
[\overline{x_2}:\alpha_{2}\overline{x_1}:\alpha_{3}\overline{x_4}:\alpha_{4}\overline{x_3}:\alpha_{5}\overline{x_6}:\alpha_{6}\overline{x_5}],
$$
where
$\alpha_{2}^{k}=\alpha_{4}^{k}=-r^{2}, \; \alpha_{3}^{k}=1, \; \alpha_{5}^{k}=re^{i\theta}=-\alpha_{6}^{k}.$
In particular, if $k \geq 2$ is even, then $C^{(k)}_{r,\theta}$ is pseudo-real.
\end{lemm}
\begin{proof}
It can be observed that the only extended M\"obius transformation keeping invariant the collection $\{\infty,0,1,-r^{2},re^{i\theta},-re^{i\theta}\}$ is $T(z)=-r^{2}/\overline{z}$.
So, following the same arguments as in the proof of \cite[Corollary 9]{GHL}, one may observe that every anticonformal automorphism of $C^{(k)}_{r,\theta}$ must have the form
{\small
$$
Q([x_1:x_2:x_3:x_4:x_5:x_6])=
[\overline{x_2}:\alpha_{2}\overline{x_1}:\alpha_{3}\overline{x_4}:\alpha_{4}\overline{x_3}:\alpha_{5}\overline{x_6}:\alpha_{6}\overline{x_5}]=[y_{1}:y_{2}:y_{3}:y_{4}:y_{5}:y_{6}].
$$
}
Now, the condition $y_{1}^{k}+y_{2}^{k}+y_{3}^{k}=0$ is equivalent to have
$\overline{\alpha_{2}^{k}} x_{1}^{k}+x_{2}^{k}+\overline{\alpha_{3}^{k}} x_{4}^{k}=0$, from which we must have $\alpha_{3}^{k}=1$ and $\alpha_{2}^{k}=-r^{2}$. Now, the condition $-r^{2}y_{1}^{k}+y_{2}^{k}+y_{4}^{k}=0$ is equivalent to have $\alpha_{4}^{k}=-r^{2}$.  The condition $r e^{i\theta}y_{1}^{k}+y_{2}^{k}+y_{5}^{k}=0$ is equivalent to have $\alpha_{5}^{k}=r e^{i\theta}$. The condition $-r e^{i\theta}y_{1}^{k}+y_{2}^{k}+y_{6}^{k}=0$ is equivalent to have $\alpha_{6}^{k}=-r e^{i\theta}$.
Let us observe that 
$$Q^{2}[x_{1}:x_{2}:x_{3}:x_{4}:x_{5}:x_{6}]=[\overline{\alpha_{2}} x_{1}:\alpha_{2} x_{2}: \alpha_{3} \overline{\alpha_{4}} x_{3}: \alpha_{4} \overline{\alpha_{3}} x_{4}: \alpha_{5} \overline{\alpha_{6}} x_{5}: \alpha_{6} \overline{\alpha_{5}} x_{6}].$$
So, for having $Q^{2}=I$, we need to have that $\alpha_{2} \in {\mathbb R}$ and  $\alpha_{2}^{k}=-r^{2}<0$. This is impossible for $k \geq 2$ even.
\end{proof}

\begin{ex}
Each curve $C^{(2)}_{r,\theta}$ admits the following anticonformal automorphism of order $4$
\begin{equation}\label{orden4}
\widehat{\eta}([x_1:x_2:x_3:x_4:x_5:x_6])=
[\overline{x_2}:ir\;\overline{x_1}:\overline{x_4}:ir\;\overline{x_3}:\sqrt{r}\;e^{i\theta/2}\overline{x_6}:i\sqrt{r}\;e^{i\theta/2}\overline{x_5}].
\end{equation}
\end{ex}

Summarizing all the above is the following result, which generalizes the one obtained in \cite{Hidalgo}.

\begin{theo}\label{teo00}
Let $\theta$ be so that $e^{i\theta} \notin \{\pm 1, \pm i\}$, let $r>1$ be so that $r \neq r_{\theta}$ and $k \geq 2$ even.
Then $C^{(k)}_{r,\theta}$  is a pseudo-real non-hyperelliptic Riemann surface of genus $g=1+(2k-3)k^{4}$ with ${\rm Aut}(C_{r,\theta})=H$. So, its field of moduli is a subfield of ${\mathbb R}$ and it cannot be defined over ${\mathbb R}$ (in particular, it is not definable over its field of moduli).
\end{theo}

\begin{rem}
The statement provided above is slightly different than the one provided in \cite{Hidalgo,Hidalgo2}, where it was assumed $k=2$. In \cite{Hidalgo}, it was missing the restriction that $r \neq r_{\theta}$ (it was corrected in \cite{Hidalgo2}). 
\end{rem}

%%%%%%%%%%%%%%%%%
%%%%%%%%%%%%%%%%%
\section{Computing the field of moduli of $C^{(k)}_{r,\theta}$}\label{Sec:computation}
The arguments, provided in the above section, 
 only permitted to see that the field of moduli of a curve $C^{(k)}_{r,\theta}$ was a subfield of ${\mathbb R}$ and that the curve was not definable over ${\mathbb R}$. It did not show how to compute explicitly the field of moduli and neither a minimal field of definition. In this section, we proceed to compute it explicitly.

 %%%%%%%%%%%%%%%
\subsection{A preliminary result}
In Lemma \ref{lemita1}, we have observed a necessary and sufficient condition for $\sigma \in {\rm Aut}({\mathbb C}/{\mathbb Q})$ to be in ${\rm Aut}({\mathbb C}/{\mathcal M}(C^{(k)}(\lambda_{1},\lambda_{2},\lambda_{3})))$, for each $(\lambda_{1},\lambda_{2},\lambda_{3}) \in \Omega$. In the particular case $\lambda_{1}=-r^{2}$, $\lambda_{2}=re^{i\theta}$ and $\lambda_{3}=-re^{i\theta}$, where $e^{i\theta} \notin \{ \pm 1, \pm i\}$, $r>1$, $r \neq r_{\theta}$ and $k \geq 2$ even, we make this description more explicitly.

\begin{prop}\label{teo0}
Let $\theta$ be so that $e^{i\theta} \notin \{ \pm 1, \pm i\}$, let $r>1$ be so that $r \neq r_{\theta}$ and $k \geq 2$ even. Set $\lambda=-r^{2}$ and 
$\mu=re^{i \theta}$. 
Then, 
$\sigma \in {\rm Aut}({\mathbb C}/{\mathbb Q})$ belongs to ${\rm Aut}({\mathbb C}/{\mathcal M}(C^{(k)}_{r,\theta}))$ if and only if any of the situations shown in Table \ref{tabla} holds for $\sigma(\lambda)$ and $\sigma(\mu)$. The third column provides the unique M\"obius transformation $T_{\sigma}$ sending the set $\{\infty, 0, 1, \lambda, \mu, -\mu\}$ onto the set $\{\infty, 0, 1, \sigma(\lambda), \sigma(\mu), -\sigma(\mu)\}$.
\end{prop}
{\small
\begin{equation}\label{tabla}
\begin{tabular}{|c|c|c|c|}\hline
Case & $\sigma(\lambda)$ & $\sigma(\mu)$ & $T_{\sigma}(z)$ \\\hline 
(1) & $\lambda$ & $\pm \mu$ &  $z$\\\hline
(2) & $1/\lambda$ & $\pm \mu/\lambda$ & $\sigma(\lambda)z$\\\hline
(3) & $1/\lambda$ & $\pm 1/\mu$ & $1/z$\\\hline
(4) & $\lambda$ & $\pm \lambda/\mu$ & $\sigma(\lambda)/z$\\\hline
(5) & $\left(\dfrac{1-\mu}{1+\mu}\right) \left(\dfrac{\lambda+\mu}{\lambda-\mu}\right)$ & $\dfrac{1-\mu}{1+\mu}$ & $\sigma(\mu) \left( \dfrac{z+\mu}{z-\mu}\right)$\\\hline
(6) & $\left(\dfrac{1+\mu}{1-\mu}\right) \left(\dfrac{\lambda-\mu}{\lambda+\mu}\right)$ & $\dfrac{\lambda-\mu}{\lambda+\mu}$ & $\sigma(\mu) \left( \dfrac{z+\mu}{z-\mu}\right)$\\\hline
(7) & $\left(\dfrac{1-\mu}{1+\mu}\right) \left(\dfrac{\lambda+\mu}{\lambda-\mu}\right)$ & $\dfrac{\mu-1}{1+\mu}$ & $-\sigma(\mu) \left( \dfrac{z+\mu}{z-\mu}\right)$\\\hline
(8) & $\left(\dfrac{1+\mu}{1-\mu}\right) \left(\dfrac{\lambda-\mu}{\lambda+\mu}\right)$ & $\dfrac{\mu-\lambda}{\mu+\lambda}$ & $-\sigma(\mu) \left( \dfrac{z+\mu}{z-\mu}\right)$\\\hline
(9) & $\left(\dfrac{1+\mu}{\mu-1}\right) \left(\dfrac{\lambda-\mu}{\lambda+\mu}\right)$ & $\dfrac{1+\mu}{\mu-1}$ & $\sigma(\mu) \left( \dfrac{z-\mu}{z+\mu}\right)$\\\hline
(10) & $\left(\dfrac{1-\mu}{1+\mu}\right) \left(\dfrac{\lambda+\mu}{\lambda-\mu}\right)$ & $\dfrac{\lambda+\mu}{\lambda-\mu}$ & $\sigma(\mu) \left( \dfrac{z-\mu}{z+\mu}\right)$\\\hline
(11) & $\left(\dfrac{1+\mu}{1-\mu}\right) \left(\dfrac{\lambda-\mu}{\lambda+\mu}\right)$ & $\dfrac{\mu+1}{\mu-1}$ &$-\sigma(\mu) \left( \dfrac{z-\mu}{z+\mu}\right)$\\\hline
(12) & $\left(\dfrac{1-\mu}{1+\mu}\right) \left(\dfrac{\lambda+\mu}{\lambda-\mu}\right)$ & $\dfrac{\mu+\lambda}{\mu-\lambda}$ & $-\sigma(\mu) \left( \dfrac{z-\mu}{z+\mu}\right)$\\\hline
\end{tabular}
\end{equation}
}

The proof of the above result will be provided in Section \ref{Sec:pruebapropo}.

%%%%%%%%%%%%
\begin{rem}
\begin{enumerate}
\item In cases (5)-(12), in the above table, one has that $\sigma(\lambda) \notin {\mathbb R}$ since 
$${\rm Im}\left[ \left( \dfrac{1-\mu}{1+\mu} \right) \left( \dfrac{\lambda+\mu}{\lambda-\mu}\right)\right]=\dfrac{2\lambda(\lambda+1)\sin(\theta)\left( \sqrt{-\lambda}-1\right)}{|1+\mu|^{2} |\lambda-\mu|^{2}} \neq 0.$$

\item If there are  $\sigma_{1}, \sigma_{2} \in {\rm Aut}({\mathbb C}/{\mathcal M}(C^{(k)}_{r,\theta}))$ so that $\sigma_{1}$ satisfies case (2) and $\sigma_{2}$ satisfies case (3), then $\sigma_{1}\sigma_{2}$ satisfies case (4).

\item If there are  $\sigma_{1}, \sigma_{2} \in {\rm Aut}({\mathbb C}/{\mathcal M}(C^{(k)}_{r,\theta}))$ so that $\sigma_{1}$ satisfies case (2) and $\sigma_{2}$ satisfies case (4), then $\sigma_{1}\sigma_{2}$ satisfies case (3).

\item If there are  $\sigma_{1}, \sigma_{2} \in {\rm Aut}({\mathbb C}/{\mathcal M}(C^{(k)}_{r,\theta}))$ so that $\sigma_{1}$ satisfies case (3) and $\sigma_{2}$ satisfies case (4), then $\sigma_{1}\sigma_{2}$ satisfies case (2).

\end{enumerate}
\end{rem}

%%%%%%%%%%%%%%%%%%
\subsection{Explicit computation of the field of moduli}
For each pair $\nu, \eta \in {\mathbb C}-\{0,1\}$, $\nu \neq \pm\eta$, and $k \geq 2$ even, we will set $D_{\nu,\eta}=C^{(k)}(\nu,\eta,-\eta)$.
Note that $C^{(k)}_{r,\theta}=D_{-r^{2},re^{i\theta}}$.

\begin{theo}\label{teo1}
Let $\theta$ be so that $e^{i\theta} \notin \{\pm 1, \pm i\}$, let $r>1$ be so that $r \neq r_{\theta}$, and let $k \geq 2$ even.
If $r^{4} \in {\mathbb Q}$ and there is no $\sigma \in {\rm Aut}({\mathbb C}/{\mathbb Q})$ so that $\sigma(re^{i\theta})=-re^{i\theta}$ (in particular, that $e^{i\theta}$ is algebraic), then 
 ${\mathcal M}(C^{(k)}_{r,\theta})={\mathbb Q}(r^{2},re^{i\theta}) \cap {\mathbb R}={\mathbb Q}(r^{2},r\cos(\theta))$. So ${\mathbb Q}(r^{2},re^{i\theta})$ is a minimal field of definition of $C^{(k)}_{r,\theta}$.
\end{theo}
\begin{proof}
Let $\theta$ be so that $e^{i\theta} \notin \{\pm 1, \pm i\}$ and  let $r>1$ be so that $r \neq r_{\theta}$. So, 
$C^{(k)}_{r,\theta}=D_{\lambda,\mu}$, where $\lambda=-r^{2}$ and $\mu=re^{i\theta}$.
As ${\mathbb Q}(\lambda,\mu)$ is a field of definition of $D_{\lambda,\mu}$ and its field of moduli is a real subfield, it follows that ${\mathcal M}(D_{\lambda,\mu})<{\mathbb Q}(\lambda,\mu) \cap {\mathbb R}={\mathbb Q}(\lambda,\sqrt{|\lambda|}\cos(\theta))={\mathbb Q}(r^{2},r\cos(\theta))$.

Since $\lambda^{2}=r^{4} \in {\mathbb Q}$, it follows that for each $\sigma \in {\rm Aut}({\mathbb C}/{\mathbb Q})$ it holds that $\sigma(\lambda)=\pm \lambda$. As $\lambda \in {\mathbb R}$, $\lambda<-1$, it follows (by Proposition \ref{teo0}) that $\sigma \in {\rm Aut}({\mathbb C}/{\mathcal M}(D_{\lambda,\mu}))$ if and only if any one of the following cases holds. 
 \begin{enumerate}
\item[(a)] $\sigma(\lambda)=\lambda$ and $\sigma(\mu)=\mu$.
\item[(b)] $\sigma(\lambda)=\lambda$ and $\sigma(\mu)=-\mu$.
\item[(c)] $\sigma(\lambda)=\lambda$ and $\sigma(\mu)=\lambda/\mu$. 
\item[(d)] $\sigma(\lambda)=\lambda$ and $\sigma(\mu)=-\lambda/\mu$. 
\end{enumerate}

As we are assuming that there is no $\sigma \in {\rm Aut}({\mathbb C}/{\mathbb Q})$ so that $\sigma(\mu)=-\mu$,  condition (b) above does not happen. Also, there are not two elements $\sigma_{1}, \sigma_{2}  \in {\rm Aut}({\mathbb C}/{\mathcal M}(D_{\lambda,\mu}))$ such that $\sigma_{1}$ satisfies (c) and $\sigma_{2}$ satisfies (d). In fact, if that happens, then $\sigma_{1}\sigma_{2} \in {\rm Aut}({\mathbb C}/{\mathcal M}(D_{\lambda,\mu}))$ satisfies (b), a contradiction.
Let us now consider the subgroup
$$H=\{ \sigma \in {\rm Aut}({\mathbb C}/{\mathcal M}(D_{\lambda,\mu})): \sigma(\mu)=\mu\}.$$

We have from (a) above that ${\mathbb Q}(\lambda,\mu)={\rm Fix}(H)$. 
Let $\sigma_{1}, \sigma_{2} \in {\rm Aut}({\mathbb C}/{\mathcal M}(D_{\lambda,\mu}))$ both satisfying (c) or both satisfying (d), then $\sigma_{1}\sigma_{2}(\lambda)=\lambda$ and $\sigma_{1}\sigma_{2}(\mu)=\mu$. In particular, 
 ${\mathbb Q}(\lambda,\mu)$ is an extension of ${\mathcal M}(C_{\lambda,\mu})$ of degree at most two. As ${\mathcal M}(D_{\lambda,\mu})<{\mathbb R}$ and ${\mathbb Q}(\lambda,\mu)$ is not contained in ${\mathbb R}$, we obtain that the extension is of degree two.
All the above asserts that ${\mathbb Q}(\lambda,\mu)$ is a minimal field of definition of $C^{(k)}_{r,\theta}=D_{\lambda,\mu}$ and that ${\mathcal M}(C^{(k)}_{r,\theta})={\mathbb Q}(\lambda,\mu) \cap {\mathbb R}={\mathbb Q}(r^{2},r\cos(\theta))$.
\end{proof}

We wonder if, in Theorem \ref{teo1}, the condition $r^{4} \in {\mathbb Q}$ can be deleted.

\begin{conj}
Let $\theta$ be so that $e^{i\theta} \notin \{\pm 1, \pm i\}$, let $r>1$ be so that $r \neq r_{\theta}$, and let $k \geq 2$ even.
If there is no
$\sigma \in {\rm Aut}({\mathbb C}/{\mathbb Q})$ so that $\sigma(re^{i\theta})=-re^{i\theta}$, then 
 ${\mathcal M}(C^{(k)}_{r,\theta})={\mathbb Q}(r^{2},re^{i\theta}) \cap {\mathbb R}={\mathbb Q}(r^{2},r\cos(\theta))$; so ${\mathbb Q}(r^{2},re^{i\theta})$ is a minimal field of definition for it.
\end{conj}

%%%%%%%%%%%%%%%%%%%%
\subsection{Some remarks}
In the above, we have assume that $r \neq r_{\theta}$ and that $e^{i\theta} \notin \{\pm 1, \pm i\}$. Below, we provide some information on the left cases.

 %%%%%%%%%%%%%%%%%%%%
\subsubsection{The case $r=r_{\theta}$, $e^{i\theta} \notin \{\pm 1, \pm i\}$}
If $r=r_{\theta}$, $e^{i\theta} \notin \{\pm 1, \pm i\}$ and $k \geq 2$ even, then the orbifold ${\mathcal O}=C^{(k)}_{r_{\theta},\theta}/H$ admits the conformal involution $$T(z)=r_{\theta}e^{i\theta}\left(\frac{z+r_{\theta}e^{i\theta}}{z-r_{\theta}e^{i\theta}}\right),$$
whose fixed points are
$$-\left(\sqrt{2}-1\right)r_{\theta}e^{i\theta}, \quad \left(\sqrt{2}+1\right)r_{\theta}e^{i\theta}.$$
Direct computations permit us to see that $T$ is the only non-trivial conformal automorphism of ${\mathcal O}$ (the computations are similar to those done by Earle in \cite{Earle2}). The transformation $T$ lifts to a conformal involution $\widehat{T}$ of the curve $C^{(k)}_{r_{\theta},\theta}$ \cite{GHL} and $C^{(k)}_{r_{\theta},\theta}/{\rm Aut}(C^{(k)}_{r_{\theta},\theta})=C^{(k)}_{r_{\theta},\theta}/\langle H, \widehat{T}\rangle={\mathcal O}/\langle T \rangle$ is an orbifold with signature $(0;2,2,2,2,2)$. As a consequence of \cite[Lemma 5.1]{Huggins2}, the curve $C^{(k)}_{r_{\theta},\theta}$ can be defined over its field of moduli. The computation of such a field of moduli can be done in a similar way as worked in \cite{HR}.

%%%%%%%%%%%%%%%%%
\subsubsection{The case $e^{i\theta} \in \{\pm 1, \pm i\}$}
If $r>1$, then (i) the curve $C^{(2)}_{r,0}$ admits the anticonformal involution $T_{2}([x_{1}:\cdots:x_{6}])=[\overline{x_{1}}:\cdots: \overline{x_{6}}]$
and (ii) the curve $C^{(2)}_{r,\pm \pi/2}$ admits the anticonformal involution $T_{1}([x_{1}:\cdots:x_{6}])=[\overline{x_{1}}:\overline{x_{2}}:\overline{x_{3}}:\overline{x_{4}}:\overline{x_{6}}:\overline{x_{5}}]$ (so they can also be defined over ${\mathbb R}$ by some suitable corresponding real curves). If $\widehat{\eta}$ is the order four anticonformal automorphism as in \eqref{orden4}, then we set $A=\widehat{\eta} \circ T_{1}$ and $B=\widehat{\eta} \circ T_{2}$. In this way, we have that $A$ defines a conformal automorphism of order $4$ on $C^{(2)}_{r,\pm \pi/2}$ and that $B$ defines a conformal involution on $C^{(2)}_{r,0}$.
We may see that ${\rm Aut}(C^{(2)}_{r,\pm \pi/2})=\langle H,B\rangle$ and that, for $r \neq r_{0}=1+\sqrt{2}$, ${\rm Aut}(C^{(2)}_{r,0})=\langle H,A\rangle$. Now, $C^{(2)}_{r,\pm \pi/2}/\langle H, A\rangle$ has signature $(0;2,2,4,4)$ and that $C_{r,0}/\langle H, B\rangle$ has signature $(0;2,2,2,2,2)$. 
As a consequence of \cite[Lemma 5.1]{Huggins2}, one has that $C^{(2)}_{r,0}$, for $r \neq r_{0}$, is definable over its field of moduli. If $r=r_{0}$, then $C^{(2)}_{r_{0},0}$ as an extra automorphism so that $C^{(2)}_{r_{0},0}/{\rm Aut}(C^{(2)}_{r_{0},0})$ has signature $(0;2,2,2,4)$; so it is also definable over its field of moduli \cite{AQ}.
For the case of $C^{(2)}_{r,\pm \pi/2}$, one may find a similar result as Proposition \ref{teo0} below.

%%%%%%%%%%%%%%%%%%%%%%%
%%%%%%%%%%%%%%%%%%%%%%%%
\section{A couple of examples}\label{sec:examples}
Next, we provide two examples, the first one on which we may apply Theorem \ref{teo1} and the second one where the conclusions of Theorem \ref{teo1} do not hold.

\subsection{Example 1}
Let $r=2$, $k \geq 2$ even, and $\theta=2\pi/p$, where $p \geq 3$ is a prime. We may see that  $r_{\theta}=\sqrt{1+\cos(\theta)^{2}}+\cos(\theta) \neq 2$.
For every $\sigma \in {\rm Aut}({\mathbb C}/{\mathbb Q})$, it holds that 
$\sigma(2e^{i\pi/p}) \neq -2e^{i\pi/p}$. In this way, Theorem \ref{teo1} states that ${\mathcal M}(C^{(k)}_{2,2\pi/p})={\mathbb Q}(\cos(2\pi/p))$ and that a minimal field of definition of $C^{(k)}_{2,2\pi/p}$ is ${\mathbb Q}(e^{2 \pi i/p})$. In particular, if  $p=3$, then ${\mathcal M}(C^{(k)}_{2,2\pi/3})={\mathbb Q}$ and a minimal field of definition is ${\mathbb Q}\left(i\sqrt{3}\right)$, and if we take $p=5$, then ${\mathcal M}(C^{(k)}_{2,2\pi/5})={\mathbb Q}\left(\sqrt{5}\right)$ and a minimal field of definition is ${\mathbb Q}\left(i\sqrt{5+\sqrt{5}} \right)$.

\subsection{Example 2}\label{ejemplo2}
Let $r=2$, $k=2$, and $\theta=\pi/4$. We set $C_{r,\theta}=C^{(2)}_{r,\theta}$.
In this case, there is a $\sigma \in {\rm Aut}({\mathbb C}/{\mathbb Q})$ with $\sigma(2e^{i\pi/4})=-2e^{i\pi/4}$ (so one of the hypothesis of Theorem \ref{teo1} does not hold). The curve $C_{r,\pi/4}$ is already defined over ${\mathbb Q}\left(e^{\pi i/4}\right)={\mathbb Q}\left(\sqrt{2}\left(1+i\right)\right)$, which is a degree four Galois extension of ${\mathbb Q}$. As every $\sigma \in {\rm Aut}({\mathbb Q}(e^{\pi i/4})/{\mathbb Q})$ satisfies the conditions of Proposition \ref{teo0}, we may see that ${\mathcal M}(C_{2,\pi/4})={\mathbb Q} \neq {\mathbb Q}(\cos(\pi/4))$ (so, the conclusion of Theorem \ref{teo1} does not hold in this case). We know that $C_{2,\pi/4}$ is definable over a degree two extension of ${\mathbb Q}$. As we know that $C_{2,\pi/4}$ cannot be definable over ${\mathbb R}$, we only need to consider the extensions of the form ${\mathbb Q}(i\sqrt{d})$, where $d \geq 1$ is an integer. We may consider the chain of degree two Galois extensions
$${\mathbb Q}<{\mathbb Q}(i\sqrt{2})<{\mathbb Q}(e^{\pi i/4}).$$

We have that ${\rm Aut}({\mathbb Q}(e^{\pi i/4})/{\mathbb Q}(i\sqrt{2}))=\langle \tau \rangle \cong {\mathbb Z}_{2}$, where $\tau(\omega)=\omega^{3}$ and $\omega=e^{\pi i/4}$.  As $\tau(i\sqrt{2})=i\sqrt{2}$, we may choose an extension of $\tau$ to ${\rm Aut}({\mathbb C}/{\mathbb Q})$ so that $\tau(i)=-i$ (see $\rho$ below); so $\tau(\sqrt{2})=-\sqrt{2}$.
An isomorphism $f_{\tau}:C_{2,\pi/4} \to C_{2,\pi/4}^{\tau}=C_{2,3\pi/4}$ must satisfies that
$\pi \circ f_{\tau} =T \circ \pi,$
where $T(z)=-4/z$ (by Proposition \ref{teo0}). It follows that necessarily $f_{\tau}$ should have the form
$$f_{\tau}([x_{1}:x_{2}:x_{3}:x_{4}:x_{5}:x_{6}])$$
$$||$$
$$[\pm x_{2}:\pm 2ix_{1}:\pm x_{4}:\pm 2ix_{3}:\pm \sqrt{2}\omega^{3/2} x_{5}:\pm i\sqrt{2}\omega^{3/2} x_{6}].$$

Unfortunately, none of the above isomorphisms is defined over ${\mathbb Q}(e^{\pi i/4})$, they are defined over ${\mathbb Q}(e^{\pi i/8})$, so we cannot apply Weil's Galois descent theorem with respect to the Galois extension ${\mathbb Q}(i\sqrt{2})<{\mathbb Q}(e^{\pi i/4})$.

Let us now consider the degree $8$ Galois extension of ${\mathbb Q}$ given by ${\mathbb Q}(e^{\pi i/8})$. Set $\eta=e^{\pi i/8}$ and consider the chain of Galois extensions
$${\mathbb Q}<{\mathbb Q}(i\sqrt{2})<{\mathbb Q}(e^{\pi i/8}).$$

In this case, ${\rm Aut}({\mathbb Q}(e^{\pi i/8})/{\mathbb Q}(i\sqrt{2}))=\langle \rho \rangle \cong {\mathbb Z}_{4}$, where
$\rho(\eta)=\eta^{3}$. We have that $\rho(\omega)=\rho(\eta^{2})=\eta^{6}=\omega^{3}$, so $\rho$ is an extension of $\tau$. In particular, $C_{\lambda,\mu}^{\rho}=C_{\lambda,\mu}^{\tau}$. So, we may use $f_{\rho}=f_{\tau}$ by choosing the plus/minus signs, for instance
$$f_{\rho}([x_{1}:x_{2}:x_{3}:x_{4}:x_{5}:x_{6}])$$
$$||$$
$$[x_{2}:\pm 2ix_{1}:x_{4}:\pm 2ix_{3}: \sqrt{2}\omega^{3/2} x_{5}: i\sqrt{2}\omega^{3/2} x_{6}].$$

Now, we may try to use Weil's Galois descent theorem, with respect to the previous degree $4$ Galois extension. For it, to satisfy Weil's conditions, we need to define 
$$f_{\rho^{2}}=f_{\rho}^{\rho} \circ f_{\rho}, \; 
and \; f_{\rho^{3}}=f_{\rho^{2}}^{\rho} \circ f_{\rho}=f_{\rho}^{\rho^{2}} \circ f_{\rho}^{\rho}\circ f_{\rho}.$$

Next, one may check that the following equality holds
$$I=f_{\rho^{4}}=f_{\rho^{3}}^{\rho} \circ f_{\rho}=f_{\rho}^{\rho^{3}}\circ f_{\rho}^{\rho^{2}} \circ f_{\rho}^{\rho}\circ f_{\rho}.$$

It follows that ${\mathbb Q}(i \sqrt{2})$ is a minimal field of definition of $C_{2,\pi/4}$.

%%%%%%%%%%%%%%%%%%%%%%%%
%%%%%%%%%%%%%%%%%%%%%%%%
\section{Proof of Proposition \ref{teo0}}\label{Sec:pruebapropo}
Let $\theta$ be so that $e^{i\theta} \notin \{\pm 1, \pm i\}$ and  let $r>1$ be so that $r \neq r_{\theta}$. We set $\lambda=-r^{2}$ and $\mu=re^{i\theta}$ and consider the curve $D_{\lambda,\mu}=C^{(k)}_{r,\theta}$.
Let $\sigma \in {\rm Aut}({\mathbb C}/{\mathbb Q})$.
We want to determine when $\sigma \in {\rm Aut}({\mathbb C}/{\mathcal M}(D_{\lambda,\mu}))$, i.e., when 
the two curves $D_{\lambda,\mu}^{\sigma}$ and $D_{\lambda,\mu}$ are  conformallly equivalent Riemann surfaces. This happens 
if and only if there is a conformal isomorphism $f_{\sigma}:D_{\lambda,\mu} \to D_{\lambda,\mu}^{\sigma}=D_{\sigma(\lambda),\sigma(\mu)}$. As $\pi^{\sigma}=\pi$, the above is equivalent to have a M\"obius transformation $T_{\sigma}$ so that $\pi \circ f_{\sigma}= T_{\sigma} \circ \pi$, which is also equivalent to have a M\"obius transformation $T_{\sigma}$ so  that 
$$\{\infty,0,1,\lambda,\mu,-\mu\} \stackrel{T_{\sigma}}{\mapsto} \{\infty,0,1,\sigma(\lambda),\sigma(\mu),-\sigma(\mu)\}.$$

As already noticed in the proof of Lemma \ref{trivialgroup}, the only subsets of cardinality $4$ inside $\{\infty,0,1,\lambda,\mu,-\mu\}$ contained in a generalized circle are given by the following three ones:
$$\{\infty,0,1,\lambda\}, \quad \{\infty,0, \mu,-\mu\}, \quad \{1,\lambda,\mu, -\mu\}.$$

Moreover, the cross-ratios of these three subsets are non-equivalent under ${\mathbb G}$.
As M\"obius transformations send generalized circles to generalized circles, it follows that inside the collection $\{\infty,0,1,\sigma(\lambda),\sigma(\mu),-\sigma(\mu)\}$ the only subsets of cardinality $4$ contained in a generalized circle are given by the following two ones:
$$\{T_{\sigma}(\infty),T_{\sigma}(0),T_{\sigma}(1),T_{\sigma}(\lambda)\}, \; \{T_{\sigma}(\infty),T_{\sigma}(0),T_{\sigma}(\mu),T_{\sigma}(-\mu)\}, \;
\{T_{\sigma}(1),T_{\sigma}(\lambda),T_{\sigma}(\mu),T_{\sigma}( -\mu)\}.$$

Since M\"obius transformations preserve cross ratios,
the cross-ratios of these three subsets are non-equivalent under the action of ${\mathbb G}$.
In this way, we should also have that
$$[T_{\sigma}(\infty),T_{\sigma}(0),T_{\sigma}(1),T_{\sigma}(\lambda)]=[\infty,0,1,\lambda]=\lambda$$
$$[T_{\sigma}(\infty),T_{\sigma}(0),T_{\sigma}(\mu),T_{\sigma}(-\mu)]=[\infty,0,\mu,-\mu]=-1$$
$$[T_{\sigma}(1),T_{\sigma}(\lambda),T_{\sigma}(\mu),T_{\sigma}( -\mu)]=[1,\lambda,\mu,-\mu].$$

Since
$[\infty,0,\sigma(\mu),-\sigma(\mu)]=-1,$ we may see that  
$$\{\infty,0,\mu,-\mu\} \stackrel{T_{\sigma}}{\mapsto} \{\infty,0,\sigma(\mu),-\sigma(\mu)\}$$
and, in particular, that
$\{1,\lambda\} \stackrel{T_{\sigma}}{\mapsto} \{1,\sigma(\lambda)\}.$

Moreover, as seen at the end of Section \ref{cross-ratio}, we only have the following two possibilities.
$$(1) \quad \{\infty,0\} \stackrel{T_{\sigma}}{\mapsto}\{\infty,0\}, \quad \{\mu,-\mu\} \stackrel{T_{\sigma}}{\mapsto}\{\sigma(\mu),-\sigma(\mu)\}, \quad \{1,\lambda\} \stackrel{T_{\sigma}}{\mapsto} \{1,\sigma(\lambda)\}.$$
$$(2) \quad \{\infty,0\} \stackrel{T_{\sigma}}{\mapsto}\{\sigma(\mu),-\sigma(\mu)\}, \quad \{\mu,-\mu\} \stackrel{T_{\sigma}}{\mapsto}\{\infty,0\}, \quad \{1,\lambda\} \stackrel{T_{\sigma}}{\mapsto} \{1,\sigma(\lambda)\}.$$

%%%%%%%%
\subsection{}
If we are in case (1), then we obtain the following possibilities. 
\begin{enumerate}
\item[(1.1)] If $T_{\sigma}(\infty)=\infty$ and $T(0)=0$, then $T_{\sigma}(z)=rz$, for a suitable $r \in {\mathbb C}-\{0\}$. As $T_{\sigma}$ sends $\{1,\lambda\}$ onto $\{1,\sigma(\lambda)\}$, we have two possibilities.
\begin{enumerate}
\item If $T_{\sigma}(1)=1$, then $T_{\sigma}(z)=z$. It follows that 
$$\sigma(\lambda)=\lambda, \quad \sigma(\mu)=\pm \mu.$$
\item If $T_{\sigma}(1)=\sigma(\lambda)$, then $T_{\sigma}(z)=\sigma(\lambda)z$. It follows that 
$$\sigma(\lambda)=1/\lambda, \quad \sigma(\mu)=\pm \mu/\lambda.$$
\end{enumerate}

\item[(1.2)] If $T_{\sigma}(\infty)=0$ and $T(0)=\infty$, then $T_{\sigma}(z)=r/z$, for a suitable $r \in {\mathbb C}-\{0\}$. Again, as $T_{\sigma}$ sends $\{1,\lambda\}$ onto $\{1,\sigma(\lambda)\}$, we have two possibilities.

\begin{enumerate}
\item If $T_{\sigma}(1)=1$, then $T_{\sigma}(z)=1/z$. It follows that 
$$\sigma(\lambda)=1/\lambda, \quad \sigma(\mu)=\pm 1/\mu.$$

\item If $T_{\sigma}(1)=\sigma(\lambda)$, then $T_{\sigma}(z)=\sigma(\lambda)/z$. It follows that 
$$\sigma(\lambda)=\lambda, \quad \sigma(\mu)=\pm \lambda/\mu.$$
\end{enumerate}

\end{enumerate}

%%%%%%%%%%%
\subsection{}
If we are in case (2), we obtain the following possibilities. 

\begin{enumerate}
\item[(2.1)] If $T_{\sigma}(\mu)=\infty$, so $T_{\sigma}(-\mu)=0$, then we must have that 
$$T_{\sigma}(z)=R \frac{z+\mu}{z-\mu},$$
where 
$$R=\left\{ \begin{array}{ll}
\sigma(\mu), & \mbox{if $T_{\sigma}(\infty)=\sigma(\mu)$}\\
\\
-\sigma(\mu), & \mbox{if $T_{\sigma}(\infty)=-\sigma(\mu)$}
\end{array}
\right.
$$

In particular, either one of the following holds:
$${\rm (a)} \; \{\sigma(\mu)\frac{1+\mu}{1-\mu},\sigma(\mu) \frac{\lambda+\mu}{\lambda-\mu}\}=\{1,\sigma(\lambda)\}, \quad \mbox{if $R=\sigma(\mu)$ }.$$
$${\rm (b)} \; \{-\sigma(\mu)\frac{1+\mu}{1-\mu},-\sigma(\mu) \frac{\lambda+\mu}{\lambda-\mu}\}=\{1,\sigma(\lambda)\}, \quad \mbox{if $R=-\sigma(\mu)$ }.$$

\begin{enumerate}
\item If $R=\sigma(\mu)$, then either
$$(i)\; \sigma(\mu) \dfrac{1+\mu}{1-\mu}=1, \mbox{ or } (ii)\; \sigma(\mu) \dfrac{\lambda+\mu}{\lambda-\mu}=\sigma(\lambda).$$

In case (i),  $$\sigma(\mu)=\frac{1-\mu}{1+\mu}, \quad \sigma(\lambda)=\left(\frac{1-\mu}{1+\mu}\right) \left(\dfrac{\lambda+\mu}{\lambda-\mu}\right).$$
In case (ii), 
$$\sigma(\mu) \dfrac{1+\mu}{1-\mu}=\sigma(\lambda), \quad \sigma(\mu) \dfrac{\lambda+\mu}{\lambda-\mu}=1,$$
which implies that  $$\sigma(\mu)=\frac{\lambda-\mu}{\lambda+\mu}, \quad \sigma(\lambda)=\left(\frac{1+\mu}{1-\mu}\right) \left(\dfrac{\lambda-\mu}{\lambda+\mu}\right).$$

\item If $R=-\sigma(\mu)$, then either
$$(i) \; \sigma(\mu) \dfrac{1+\mu}{1-\mu}=-1, \mbox{ or } (ii)\;  \sigma(\mu) \dfrac{\lambda+\mu}{\lambda-\mu}=-\sigma(\lambda).$$
In case (i),  
$$\sigma(\mu)=\frac{\mu-1}{1+\mu}, \quad \sigma(\lambda)=\left(\frac{1-\mu}{1+\mu}\right) \left(\dfrac{\lambda+\mu}{\lambda-\mu}\right).$$
In case (ii), 
$$\sigma(\mu) \dfrac{1+\mu}{1-\mu}=-\sigma(\lambda), \quad \sigma(\mu) \dfrac{\lambda+\mu}{\lambda-\mu}=-1,$$
which implies that  $$\sigma(\mu)=\frac{\mu-\lambda}{\lambda+\mu}, \quad \sigma(\lambda)=\left(\frac{1+\mu}{1-\mu}\right) \left(\dfrac{\lambda-\mu}{\lambda+\mu}\right).$$
\end{enumerate}

\item[(2.2)] If now $T_{\sigma}(\mu)=0$, so $T_{\sigma}(-\mu)=\infty$, then the computations are as before by interchanging the roles of $\mu$ with $-\mu$. So we obtain either one of the followings:
$$(i)\; \sigma(\mu)=\frac{1+\mu}{\mu-1}, \quad \sigma(\lambda)=\left(\frac{1+\mu}{1-\mu}\right) \left(\dfrac{\lambda-\mu}{\lambda+\mu}\right);$$
$$(ii) \; \sigma(\mu)=\frac{\lambda+\mu}{\mu-\lambda}, \quad \sigma(\lambda)=\left(\frac{1-\mu}{1+\mu}\right) \left(\dfrac{\lambda+\mu}{\lambda-\mu}\right);$$
$$(iii)\; \sigma(\mu)=\frac{1+\mu}{1-\mu}, \quad \sigma(\lambda)=\left(\frac{1+\mu}{1-\mu}\right) \left(\dfrac{\lambda-\mu}{\lambda+\mu}\right);$$
$$(iv)\; \sigma(\mu)=\frac{\lambda+\mu}{\lambda-\mu}, \quad \sigma(\lambda)=\left(\frac{1-\mu}{1+\mu}\right) \left(\dfrac{\lambda+\mu}{\lambda-\mu}\right).$$

\end{enumerate}

All the computations above permit us to see, in particular, that 
$\sigma_{0}(z)=\overline{z} \in {\rm Aut}({\mathbb C}/{\mathcal M}(D_{\lambda,\mu}))$. It satisfies that $\sigma_{0}(\lambda)=\lambda$ and that $\sigma_{0}(\mu)=-\lambda/\mu$.

%%%%%%%%%%%%%%%
%%%%%%%%%%%%%%%

%%%%%%%%%%%%%%%%%%%%%%
%%%%%%%%%%%%%%%%%%%%%%


\begin{thebibliography}{99}

\bibitem{AQ}
M. Artebani and S. Quispe. 
Fields of moduli and fields of definition of odd signature curves.
{\it Archiv der Mathematik} {\bf 99} (2012), 333--343.


\bibitem{Achq}
M. Artebani, M. Carvacho, R.A. Hidalgo,  and S. Quispe.
A tower of Riemann surfaces which cannot be defined over their field of moduli.
{\it Glasgow Mathematical Journal} {\bf 59} (2017), 379--393.

\bibitem{BHQ}
E. Badr, R. A. Hidalgo, and S. Quispe.
Riemann surfaces defined over the reals.
{\it Archive der Mathematik} {\bf 110} (2018), 351--362.

\bibitem{CGHR}
A. Carocca, V. Gonz\'alez, R. A. Hidalgo and Rub\'{\i} Rodr\'{\i}guez. 
Generalized Humbert Curves.
{\it Israel Journal of Mathematics} {\bf 64} No. 1 (2008), 165-192.


\bibitem{DE}
P. D\`ebes and M. Emsalem.
On Fields of Moduli of Curves.
{\it J. of Algebra} {\bf 211} (1999), 42-56.



\bibitem{Earle1}
C. J. Earle.
On the moduli of closed Riemann surfaces with symmetries.
{\it Advances in the Theory of Riemann Surfaces} (1971), 119-130. Ed. L.V. Ahlfors et al. 
(Princeton Univ. Press, Princeton).


\bibitem{Earle2}
C. J. Earle.
Diffeomorphisms and automorphisms of compact hyperbolic 2-orbifolds. 
Geometry of Riemann surfaces, 139Ð155, {\it London Math. Soc. Lecture Note Ser.} {\bf 368}, Cambridge Univ. Press, Cambridge, 2010.


\bibitem{GHL}
G. Gonz\'alez-Diez, R. A. Hidalgo and M. Leyton.
Generalized Fermat Curves.
{\it Journal of Algebra} {\bf 321} (2009), 1643-1660.

\bibitem{HH}
H. Hammer and F Herrlich.
A Remark on the Moduli Field of a Curve.
{\it Arch. Math.} {\bf 81} (2003), 5-10.

\bibitem{Hidalgo:JNT}
R. A. Hidalgo.
Holomorphic differentials of generalized Fermat curves.
{\it Journal of Number Theory} {\bf 217} (2020) 78--101.

\bibitem{Hidalgo}
R. A. Hidalgo.
Non-hyperelliptic Riemann surfaces with real field of moduli but not definable over the reals. 
{\it Archiv der Mathematik} {\bf 93} (2009), 219-222.

\bibitem{Hidalgo2}
R. A. Hidalgo. 
Erratum: Non-hyperelliptic Riemann surfaces with real field of moduli but not definable over the reals. 
{\it Archiv der Mathematik} {\bf 98} (2012), 449--451.


\bibitem{Hidalgo:Homology}
R. A. Hidalgo.
Homology closed Riemann surfaces.
{\it Quarterly Journal of Math.} {\bf 63} (2012), 931--952.

\bibitem{KFT}
R. A. Hidalgo.
Computing the field of moduli of the KFT family.
{\it Proyecciones Journal of Mathematics} {\bf 33} (2014), 61--75.


\bibitem{Hidalgo:FOD/FOM}
R. A. Hidalgo.
A remark on the field of moduli of Riemann surfaces.
{\it Arch. Math.} {\bf 114} (2020), 515--526.

\bibitem{HJ}
R. A. Hidalgo and P. Johnson.
Field of moduli of generalized Fermat curves of type $(k,3)$ with an application to non-hyperelliptic dessins d'enfants.
{\it Journal of Symbolic Computation} {\bf 71} (2015), 60--72.



\bibitem{HKLP}
R. A. Hidalgo, A. Kontogeorgis, M. Leyton-Alvarez and P. Paramantzoglou. 
Automorphisms of Generalized Fermat Curves. 
{\it Journal of Pure and Applied Algebra} {\bf 221} (2017), 2312--2337.


\bibitem{HR}
R. A. Hidalgo and S. Reyes-Carocca.
Fields of moduli of classical Humbert curves.
{\it Quarterly Journal of Math.} {\bf 63} (2012), 919--930.

\bibitem{Huggins2}
B. Huggins.
Fields of moduli of hyperelliptic curves.
{\it Math. Res. Lett.} {\bf 14} No. 2 (2007), 249-262.


\bibitem{Koizumi}
S. Koizumi.
Fields of moduli for polarized Abelian varieties and for curves.
{\it Nagoya Math. J.} {\bf 48} (1972), 37-55.

\bibitem{Kontogeorgis}
A. Kontogeorgis.
Field of moduli versus field of definition for cyclic covers of the projective line.
{\it J. de Theorie des Nombres de Bordeaux} {\bf 21} (2009) 679--692.

\bibitem{Matsusaka}
T. Matsusaka.
Polarized varieties, field of moduli and generalized Kummer varieties of polarized abelian varieties.
{\it Amer. J. Math.} {\bf 80} (1958), 45-82.


\bibitem{Shimura}
G. Shimura.
On the field of rationality for an abelian variety. 
{\it Nagoya Math. J.} {\bf 45} (1972), 167-178. 


\bibitem{Silhol}
R. Silhol.
Moduli problems in real algebraic geometry.
{\it Real Algebraic Geometry} (1972), 110-119. Ed. M. Coste et al. (Springer-Verlag, Berlin).



\bibitem{Weil}
A. Weil.
The field of definition of a variety.
{\it Amer. J. Math.} {\bf 78} (1956), 509-524.

\bibitem{Wo1}
Wolfart, J.
The Obvious Part of Belyi's Theorem and Riemann Surfaces with Many Automorphisms.
 In {\it Geometric Galois Actions 1. Around Grothendieck's Esquisse d'un Programme}, London Math. Soc. Lecture Note Ser. 242, Cambridge UP 1997, pp-97-112.
 
 \bibitem{Wo2}
 Wolfart, J.
 $ABC$ for polynomials, dessins d'enfants and uniformization---a survey. Elementare und analytische Zahlentheorie, 313--345, Schr. Wiss. Ges. Johann Wolfgang Goethe Univ. Frankfurt am Main, 20, Franz Steiner Verlag Stuttgart, Stuttgart, 2006.
 
\end{thebibliography}
\end{document}